\documentclass[12pt]{amsart}

\usepackage{amssymb}
\usepackage{enumerate}
\usepackage{graphicx}

\makeatletter
\@namedef{subjclassname@2010}{%
  \textup{2010} Mathematics Subject Classification}
\makeatother

\newtheorem{theorem}{Theorem}
\newtheorem{theo}{Theorem}

\theoremstyle{definition}
\newtheorem{remark}{Remark}

\newcommand{\bR}{\mathbb{R}}

\newcommand{\ol}{\overline}
\newcommand{\ul}{\underline}

\newcommand{\Int}{{\textstyle \int}}

\newcommand{\diam}{\operatorname{diam}}

\newcommand{\supp}{\operatorname{supp}}
\newcommand{\dist}{\operatorname{dist}}

\numberwithin{equation}{section}

\begin{document}
\title[On Sets of Singular Rotations]{On Sets of Singular Rotations for Translation Invariant Bases}

\author{K. A. Chubinidze}

\begin{abstract}\fontsize{9}{11pt}\selectfont
It is studied the following problem: for a given function $f$ what kind of may be a set of all rotations $\gamma$ for which $\Int f$ is not differentiable with respect to $\gamma$-rotation of a given basis $B$? In particular, for  translation invariant bases on the plane it is found the topological structure of possible sets of singular rotations.

\vskip+0.2cm

\end{abstract}

\maketitle

A mapping $B$ defined on $\mathbb{R}^n$ is said to be a \emph{differen\-tiation basis} if for every $x\in \mathbb{R}^n$, $B(x)$ is a family  of bounded measurable sets with positive measure and containing $x$, such that there exists a sequence $R_k\in B(x)$ $(k\in\mathbb{N})$ with $\lim\limits_{k\rightarrow\infty}\diam R_k = 0$.

For $f\in L(\mathbb{R}^n)$, the numbers
$$
    \overline{D}_B\left(\Int f,x\right)=\mathop{\ol{\lim}}\limits_{\substack{R\in B(x) \\ \diam R\to 0}}
        \frac{1}{|R|} \int_R f \quad \text{and} \quad
    \underline{D}{\,}_B\left(\Int f,x\right)=\mathop{\ul{\lim}}\limits_{\substack{R\in B(x) \\ \diam R\to 0}}
        \frac{1}{|R|} \int_R f
$$
are called \emph{the upper and the lower derivative,} respectively, \emph{of the integral of $f$ at a point $x$}. If the upper and the lower derivative coincide, then their combined value is called the \emph{derivative of $\Int f$ at a point $x$} and denoted by $D_B(\Int f,x)$. We say that the\emph{ basis  $B$ differentiates }$\Int f$ (or $\Int f$ is differentiable with respect to $B$) if $\ol{D}_B(\Int f,x)=\ul{D}{\,}_B(\Int f,x)=f(x)$ for almost all $x\in \bR^n$. If this is true for each $f$ in the class of functions $X$ we say that $B$ differentiates $X$.

Denote by $\textbf{I}=\textbf{I}(\mathbb{R}^n)$ the basis of intervals, i.e., the basis for which $\textbf{I}(x)$ $(x\in \mathbb{R}^n)$ consists of all open $n$-dimensional intervals containing $x$. Note that differentiation with respect to $\textbf{I}$ is called \emph{strong differentiation}.

For a basis $B$  by  $F_B$  denote the class of all functions $f\in L(\mathbb{R}^n)$ the integrals of which are differentiable with respect to $B$.

A basis $B$ is called \emph{translation invariant} (briefly, $TI$-basis) if $B(x)=\{x+R: R\in B(0)\}$ for every $x\in\mathbb{R}^n$;

Denote by $\Gamma_n$ the family of all rotations in the space $\mathbb{R}^n$.

Let $B$ be a basis in $\mathbb{R}^n$ and $\gamma\in \Gamma_n$.  The $\gamma$-\emph{rotated basis} $B$ is defined as follows
$$
B(\gamma)(x)=\{x+\gamma(R - x): R\in B(x)\}\quad (x\in\mathbb{R}^n).
$$

The set of two-dimensional rotations $\Gamma_2$ can be identified with the circumference $\mathbb{T}=\{z\in\mathbb{C}:|z|=1\}$,  if to a  rotation $\gamma$ we put into correspondence the complex number $z(\gamma)$ from $\mathbb{T}$, the argument of which is equal to the value of the angle by which the rotation about the origin takes place in the positive direction under the action of $\gamma$.

The distance $d(\gamma,\sigma)$ between points $\gamma,\sigma\in\Gamma_2$ is assumed to be equal to the length of the smallest arch of the circumference $\mathbb{T}$ connecting  points $z(\gamma)$ and $z(\sigma)$.

Let $B$ and $H$ are bases in $\mathbb{R}^n$ with $B\subset H$ and $E\subset \Gamma_n$. Let us call $E$ a  $W_{B,H}$-\emph{set} ($W_{B,H}^{+}$-\emph{set}),  if there exists a function  $f\in L(\mathbb{R}^n)$ ($f\in L(\mathbb{R}^n), f\geq 0$) such that:
1)$f\notin F_{B(\gamma)}$ for every $\gamma\in E$; and 2) $f\in F_{H(\gamma)}$ for every $\gamma\notin E$;

  Let $B$ and $H$ are bases in $\mathbb{R}^n$ with $B\subset H$ and $E\subset \Gamma_n$. Let us call $E$ an $R_{B,H}$-\emph{set} ($R_{B,H}^{+}$-\emph{set}),  if there exists a function $f\in L(\mathbb{R}^n)$ ($f\in L(\mathbb{R}^n)$, $f\geq 0$) such that:
1)~$\overline{D}_{B(\gamma)}\left(\Int f,x\right)=\infty$ almost everywhere for every $\gamma\in E$; and
2)~$f\in F_{H(\gamma)}$ for every $\gamma\notin E$.

When  $B=H$ we will use terms  $W_B$ $(W_B^{+}$, $R_B$, $R_B^{+})$-set.

\begin{remark}
It is clear that:

1) each $W_{B,H}^{+}(R_{B,H}^{+})$-set is $W_{B,H}$ $(R_{B,H})$-set;

2) each $W_{B,H}$ $(W_{B,H}^{+}$, $R_{B,H}$, $R_{B,H}^{+})$-set is $W_{B}$ $(W_{B}^{+}$, $R_B$, $R_B^{+})$-set.
\end{remark}

The definitions of  $R_{\textbf{I}(\mathbb{R}^2)}, R^{+}_{\textbf{I}(\mathbb{R}^2)}$ and $W_{\textbf{I}(\mathbb{R}^2)}$-sets were introduced in [5], [6] and [1], respectively.

\medskip

Singularities of an integral of a fixed function with respect to the collection of rotated bases $B(\gamma)$ were studied by various authors (see [1-9]). In particular, in [5]  and [1], respectively, there were proved the following results about topological structure of $R_{\textbf{I}(\mathbb{R}^2)}$ and $W_{\textbf{I}(\mathbb{R}^2)}$.

\begin{theo}
Each $R_{\mathbf{I}(\mathbb{R}^2)}$-set has $G_{\delta}$ type.
\end{theo}

\begin{theo}
Each $W_{\mathbf{I}(\mathbb{R}^2)}$-set has $G_{\delta\sigma}$ type.
\end{theo}

There are true the following generalizations of Theorems A and B.

\begin{theorem}\label{t4}
For arbitrary translation invariant basis $B$ in $\mathbb{R}^2$ each $R_B$-set has  $G_{\delta}$ type.
\end{theorem}

\begin{theorem}\label{t3}
For arbitrary  translation invariant basis $B$ in $\mathbb{R}^2$ each $W_B$-set has  $G_{\delta \sigma}$ type.
\end{theorem}

We will also prove the following result.

\begin{theorem}\label{t3}
For arbitrary bases  $B$ and $H$ in $\mathbb{R}^2$  with $B\subset H$ not more than countable union of $R_{B,H}$-sets $(R^{+}_{B,H}$-sets$)$ is $W_{B,H}$-set $(W^{+}_{B,H}$-set$)$.
\end{theorem}

\begin{proof}[Proof of Theorem $1$] Let $f \in L(\mathbb{R}^2)$. We must prove that the set
$$
W_B(f)=\{\gamma\in \Gamma_2: f\notin F_{B(\gamma)}\}
$$
is of $G_{\delta \sigma}$ type.

Without loss of generality let us assume  that $f$ is finite everywhere and $\supp f\subset (0,1)^n$.

For a basis $H$,  $x\in\mathbb{R}^2$  and $r>0$ set
\begin{align*}
l_H(f)(x)& =\mathop{\overline{\lim}}\limits_{\substack{R\in H(x)\\ \diam R\to 0}} \bigg|\frac{1}{|R|}\int_R f-f(x)\bigg|,
\\
l_H^r(f)(x)&=\sup_{\substack{R\in H(x) \\ \diam R<r}}\bigg|\frac{1}{|R|}\int_R f-f(x)\bigg|.
\end{align*}

For numbers $\varepsilon>0, \alpha\in (0,1]$, $r>0$ and $\beta\in (0,1)$ denote
$$
W_B(f,\varepsilon,\alpha)=\{\gamma\in \Gamma_2:|\{l_{B(\gamma)}(f)\geq\varepsilon\}|\geq \alpha\},
$$
$$
W_B^{r}(f,\varepsilon,\beta)=\{\gamma\in \Gamma_2:|\{l_{B(\gamma)}^r(f)>\varepsilon\}|>\beta\}.
$$

First let us prove that $W_B^{r}(f,\varepsilon,\beta)$ is an open set for any $r>0$, $\varepsilon>0$ and $\beta\in (0,1)$. Suppose $\gamma\in W_B^{r}(f,\varepsilon,\beta)$, i.e.
$$
|\{l_{B(\gamma)}^r(f)>\varepsilon\}|>\beta.
$$
If $x\in \{l_{B(\gamma)}^r(f)>\varepsilon\}$, then there is $R_x\in B(\gamma)(x)$ with $\diam R_x< r$ such that
$$
\bigg|\frac{1}{|R_x|}\int_{R_x} f-f(x)\bigg|>\varepsilon.
$$
Taking into account absolute continuity of Lebesgue integral is easy to check that performing small enough rotation of $R_x$ around the point $x$ one derives the set $R'_x$ for which
$$
\Big|\frac{1}{|R'_x|}\int_{R'_x} f-f(x)\Big|>\varepsilon.
$$
Therefore for every $x\in \{l_{B(\gamma)}^r(f)>\varepsilon\}$ we can find  $k_x\in\mathbb{N}$ such that
$$
l_{B(\gamma')}^r(f)(x)>\varepsilon\;\;\text{if}\;\;\dist (\gamma',\gamma)<1/k_x.
$$

For every $m\in\mathbb{N}$ by $A_m$ denote the set of all points from $\{l_{B(\gamma)}^r(f)>\varepsilon\}$ for which $k_x=m$. Obviously,
$$
A_1\subset A_2\subset \cdots\;\;\text{and}\;\;\bigcup_{m\in\mathbb{N}} A_m=\{l_{B(\gamma)}^r(f)>\varepsilon\}.
$$
Now, using the property of continuity of outer measure from below we can find $m\in\mathbb{N}$ for which $|A_m|_{\ast}> \beta$. The last conclusion implies that
$$
|\{l_{B(\gamma')}^r(f)>\varepsilon\}|>\beta\;\;\text{if}\;\;\dist(\gamma',\gamma)<1/m.
$$
Consequently, $W_B^{r}(f,\varepsilon,\beta)$ is an open set.

Now let us prove that $W_B(f,\varepsilon,\alpha)$ is of $G_\delta$ type for any  $\varepsilon>0$ and $\delta\in (0,1]$. Let us consider strictly increasing sequences of positive numbers $(\varepsilon_k)$ and $(\alpha_k)$ such that $\varepsilon_k\rightarrow \varepsilon$ and $\alpha_k\rightarrow \alpha$. Taking into account openness of sets $W_B^{r}(f,\varepsilon,\beta)$ it is easy to see that for every $\gamma\in W_B(f,\varepsilon,\alpha)$ and $k\in\mathbb{N}$ there is a neighbourhood $V_{\gamma,k}$ of $\gamma$ such that
$$
|\{l_{B(\gamma')}^{1/k}(f)>\varepsilon_k\}|>\alpha_k\;\;\text{if}\;\;\gamma'\in V_{\gamma,k}.
$$
Denote
$$
G_k=\bigcup_{\gamma\in W_B(f,\varepsilon,\alpha)}V_{\gamma,k}\;\;\;(k\in\mathbb{N}).
$$
Since $W_B(f,\varepsilon,\alpha)\subset G_k\;(k\in\mathbb{N})$, we have $W_B(f,\varepsilon,\alpha)\subset\bigcap\limits_{k\in\mathbb{N}}G_k$. On the other hand, if $\gamma\in \bigcap\limits_{k\in\mathbb{N}}G_k$, then
$$
|\{l_{B(\gamma)}(f)\geq\varepsilon\}|=\bigg|\bigcap\limits_{k\in\mathbb{N}} \{l_{B(\gamma)}^{1/k}(f)>\varepsilon_k\}\bigg|\geq \lim\limits_{k\rightarrow\infty}\alpha_k=\alpha.
$$
Consequently, $\gamma\in W_B(f,\varepsilon,\alpha)$. Thus $W_B(f,\varepsilon,\alpha)\supset \bigcap\limits_{k\in\mathbb{N}}G_k$. So we proved that $W_B(f,\varepsilon,\alpha)= \bigcap\limits_{k\in\mathbb{N}}G_k$, wherefrom it follows the needed conclusion.

It is easy to check that
$$
W_B(f)=\bigcup\limits_{k\in\mathbb{N}}W_B(f,1/k,1/k),
$$
wherefrom we conclude $W_B(f)$ to be of $G_{\delta\sigma}$ type.
\end{proof}

\begin{proof}[Proof of Theorem $2$]
Let $f \in L(\mathbb{R}^2)$ and $\supp f\subset (0,1)^2$. Let us prove that the set
$$
R_B(f)=\{\gamma\in \Gamma_2: \overline{D}_{B(\gamma)}\left(\Int f,x\right)=\infty \;\;\text{a.e. on}\;\;(0,1)^2\}.
$$
is of $G_{\delta}$ type. It is easy to check that this assertion implies the validity of the theorem.

For a basis $H$,  $x\in\mathbb{R}^2$  and $r>0$ set
$$
N_H^r(f)(x)=\sup_{\substack{R\in H(x) \\ \diam R<r}} \frac{1}{|R|}\int_R f.
$$

For numbers $\varepsilon>0$, $r>0$ and $\beta\in (0,1)$ denote
$$
R_B^{r}(f,\varepsilon,\beta)=\{\gamma\in \Gamma_2:|\{N_{B(\gamma)}^r(f)>\varepsilon\}|>\beta\}.
$$

First let us prove that $R_B^{r}(f,\varepsilon,\beta)$ is an open set for any $r>0$, $\varepsilon>0$ and $\beta\in (0,1)$. Suppose $\gamma\in R_B^{r}(f,\varepsilon,\beta)$, i.e.
$$
|\{N_{B(\gamma)}^r(f)>\varepsilon\}|>\beta.
$$
If $x\in \{N_{B(\gamma)}^r(f)>\varepsilon\}$, then there is $R_x\in B(\gamma)(x)$ with $\diam R_x< r$ such that
$$
\frac{1}{|R_x|}\int_{R_x} f>\varepsilon.
$$
Taking into account absolute continuity of Lebesgue integral is easy to check that performing small enough rotation of $R_x$ around the point $x$ one derives the set $R'_x$ for which
$$
\frac{1}{|R'_x|}\int_{R'_x} f>\varepsilon.
$$
Therefore for every $x\in \{N_{B(\gamma)}^r(f)>\varepsilon\}$ we can find  $k_x\in\mathbb{N}$ such that
$$
N_{B(\gamma')}^r(f)(x)>\varepsilon\;\;\text{if}\;\;\dist (\gamma',\gamma)<1/k_x.
$$

For every $m\in\mathbb{N}$ by $A_m$ denote the set of all points from $\{N_{B(\gamma)}^r(f)>\varepsilon\}$ for which $k_x=m$. Obviously,
$$
A_1\subset A_2\subset \dots\;\;\text{and}\;\;\bigcup_{m\in\mathbb{N}} A_{m}=\{N_{B(\gamma)}^r(f)>\varepsilon\}.
$$
Now, using the property of continuity of outer measure from below we can find $m\in\mathbb{N}$ for which $|A_m|_{\ast}> \beta$. The last conclusion implies that
$$
|\{N_{B(\gamma')}^r(f)>\varepsilon\}|>\beta\;\;\text{if}\;\;\dist(\gamma',\gamma)<1/m.
$$
Consequently, $R_B^{r}(f,\varepsilon,\beta)$ is an open set.

Now let us prove that $R_B(f)$ is of $G_\delta$ type. Taking into account openness of sets $R_B^{r}(f,\varepsilon,\beta)$ it is easy to see that for every $\gamma\in R_B(f)$ and $k\in\mathbb{N}$ there is a neighbourhood $V_{\gamma,k}$ of $\gamma$ such that
$$
|\{N_{B(\gamma')}^{1/k}(f)> k\}|>1-1/k\;\;\text{if}\;\;\gamma'\in V_{\gamma,k}.
$$
Denote
$$
G_k=\bigcup_{\gamma\in R_B(f)}V_{\gamma,k}\;\;\;(k\in\mathbb{N}).
$$
Since $R_B(f)\subset G_k\;(k\in\mathbb{N})$, we have $R_B(f)\subset\bigcap\limits_{k\in\mathbb{N}}G_k$. On the other hand, if $\gamma\in \bigcap\limits_{k\in\mathbb{N}}G_k$, then
\begin{multline*}
\left|\left\{\overline{D}_{B(\gamma)}\left(\Int f,x\right)=\infty \;\;\text{a.e. on}\;\;(0,1)^2\right\}\right|= \\
=\bigg|\bigcap\limits_{k\in\mathbb{N}} \{N_{B(\gamma)}^{1/k}(f)>k\}\bigg|\geq
 \lim\limits_{k\rightarrow\infty}(1-1/k)=1.
\end{multline*}
Consequently, $\gamma\in R_B(f)$. Thus $R_B(f)\supset\! \bigcap\limits_{k\in\mathbb{N}}G_k$. So we proved that $R_B(f)\!= \bigcap\limits_{k\in\mathbb{N}}G_k$. Wherefrom it follows the needed conclusion.
\end{proof}

\begin{proof}[Proof of Theorem $3$]
Let $N\subset \mathbb{N}$ be a not more than countable non-empty set  and for each $k\in N$, $E_k$ be an $R_{B,H}$-set($R^{+}_{B,H}$-set). For every $k\in N$ let us consider summable function $f_k$ with  two properties from the definition of $R_{B,H}$-set ($R^{+}_{B,H}$-set): 1) $\overline{D}_{B(\gamma)}\left(\Int f_k,x\right)=\infty$ almost everywhere for every $\gamma\in E_k$; and 2) $f_k\in F_{H(\gamma)}$ for every $\gamma\notin E_k$.  Let us consider also an arbitrary family of pairwise disjoint open squares $Q_k$ $(k\in N)$.

Denote
\begin{gather*}
g_k=\frac{f_k \chi_{Q_k}}{2^k\|f_k\|_L}\quad (k\in N),
\\
f=\sum_{k\in N} g_k.
\end{gather*}
Then we have
$$
\|f\|_L=\sum_{k\in N} \|g_k\|_L \leq \sum_{k\in N}\frac{1}{2^k}<\infty.
$$
Consequently, $f$ is summable function.

Using disjointness of squares $Q_k$ we have that for every $\gamma\in \Gamma_2$, $k\in N$  and $x\in Q_k$
$$
\overline{D}_{B(\gamma)}\left(\Int f,x\right)=\overline{D}_{B(\gamma)}\left(\Int g_k,x\right).
$$
Therefore for every $k\in N$ and  $\gamma\in E_k$
$$
\overline{D}_{B(\gamma)}\left(\Int f,x\right)=\infty\;\;\text{for a.e.}\;\;x\in Q_k.
$$
Thus,
\begin{equation}\label{1}
f\notin F_{B(\gamma)} \;\;\text{for every}\;\; \gamma\in \bigcup_{k\in N}E_k.
\end{equation}

Now take arbitrary $\gamma\notin \bigcup\limits_{k\in N}E_k$. Then $g_k\in F_{H(\gamma)}$ for every $k\in N$. Consequently, using disjointness of squares $Q_k$ we have that for every $k\in N$
$$
D_{H(\gamma)}\left(\Int f,x\right)=D_{H(\gamma)}\left(\Int g_k,x\right)=g_k(x)=f(x)
$$
for a.e. $x\in Q_k$. Thus
$$
D_{H(\gamma)}\left(\Int f,x\right)=f(x)\;\;\text{for a.e.}\;\;x\in \bigcup_{k\in N}Q_k.
$$
Now taking into account that $f(x)=0$ for every $x\notin \bigcup\limits_{k\in N}Q_k$ we write
\begin{equation}
D_{H(\gamma)}\left(\Int f,x\right)=f(x)\;\;\text{for a.e.}\;\;x\in \mathbb{R}^2.
\end{equation}

(1) and (2) implies that  $\bigcup\limits_{k\in N}E_k$ is $W_{B,H}$-set ($W^{+}_{B,H}$-set).
\end{proof}

\

\medskip

Author's address:

\vskip+0.2cm

Akaki Tsereteli State University

59, Tamar Mepe St., Kutaisi 4600

Georgia

e-mail: kaxachubi@gmail.com

\end{document}